\documentclass[11pt, a4paper]{article}
\usepackage[T1]{fontenc}
\usepackage{lmodern,microtype}
\usepackage{graphicx} 

\usepackage{epsfig}
\usepackage{caption}
\usepackage{amsfonts}
\usepackage{amssymb}
\usepackage{mathrsfs}
\usepackage{amsmath}
\usepackage{amsthm}
\usepackage{enumerate}
\usepackage{graphics}
\usepackage{pict2e}
\usepackage{cite}
\usepackage{subfigure}
\usepackage[subnum]{cases}
\usepackage{multirow}
\usepackage{comment}
\usepackage{cleveref}
\usepackage{authblk}

\makeatletter

\newcommand{\A}{\mathcal{A}}
\newcommand{\B}{\mathcal{B}}

\newcommand{\R}{\mathbb{R}} 
\newcommand{\eps}{\varepsilon}

\newcommand{\BS}{\mathit{BS}} 
\newcommand{\KS}{\mathit{KS}} 
\newcommand{\SSG}{\mathit{SS}} 
\newcommand{\HKS}{\widehat{\KS}}
\newcommand{\HSS}{\widehat{\SSG}}
\newcommand{\RP}{\R\mathrm{P}}

\DeclareMathOperator{\dist}{dist}

\begin{document}

	\newtheorem{theorem}{Theorem}
	\newtheorem{observation}[theorem]{Observation}
	\newtheorem{corollary}[theorem]{Corollary}
	\newtheorem{algorithm}[theorem]{Algorithm}
	\newtheorem{definition}[theorem]{Definition}
	\newtheorem{guess}[theorem]{Conjecture}
	\newtheorem{claim}[theorem]{Claim}
	\newtheorem{problem}[theorem]{Problem}
	\newtheorem{question}[theorem]{Question}
	\newtheorem{lemma}[theorem]{Lemma}
	\newtheorem{proposition}[theorem]{Proposition}
	\newtheorem{fact}[theorem]{Fact}

	\makeatletter
	\newcommand\figcaption{\def\@captype{figure}\caption}
	\newcommand\tabcaption{\def\@captype{table}\caption}
	\makeatother

	\newtheorem{case}[theorem]{Case}
	\newtheorem{conclusion}[theorem]{Conclusion}
	\newtheorem{condition}[theorem]{Condition}
	\newtheorem{conjecture}[theorem]{Conjecture}
	\newtheorem{criterion}[theorem]{Criterion}
	\newtheorem{example}[theorem]{Example}
	\newtheorem{exercise}[theorem]{Exercise}
	\newtheorem{notation}{Notation}
	\newtheorem{remark}[theorem]{Remark}
	\newtheorem{solution}[theorem]{Solution}

	\title{\textbf{Colouring signed analogues of Kneser, Schrijver, and Borsuk graphs}}

	\author[1]{Luis Kuffner}
	\author[2]{Reza Naserasr}
	\author[3]{Lujia Wang}
	\author[4]{\linebreak Xiaowei Yu}
	\author[2,3]{Huan Zhou}
	\author[3]{Xuding Zhu}

        \affil[1]{\small École normale supérieure, PSL University, Paris, France. {Email:\texttt{luis.kuffner.pineiro@ens.psl.eu}}}
	\affil[2]{\small Université Paris Cité, CNRS, IRIF, F-75013, Paris, France. {Emails: \texttt{\{reza, zhou\}@irif.fr}}}
	\affil[3]{\small Zhejiang Normal University, Jinhua, China. {Emails: \texttt{\{ljwang, huanzhou, xdzhu\}@zjnu.edu.cn}}}
	\affil[4]{\small Jiangsu Normal University, Xuzhou, China. {Email: \texttt{xwyu@jsnu.edu.cn}}}

	\maketitle

	\begin{abstract}

	    The Kneser signed graph $\KS(n,k)$, $k\leq n$, is the graph whose vertices are signed $k$-subsets of $[n]$ (i.e. $k$-subsets $S$ of $\{ \pm 1, \pm 2, \ldots, \pm n\}$ such that $S\cap (-S)=\emptyset$). Two vertices $A$ and $B$ are adjacent with a positive edge if $A\cap (-B)=\emptyset$ and with a negative edge if $A\cap B=\emptyset$. We prove that the balanced chromatic number of $\KS(n,k)$ is $n-k+1$.
        We then introduce the signed analogue of Schrijver graphs and show that choosing one representative from each antipodal pair of alternating signed $k$-sets gives a vertex-critical subgraph of $\KS(n,k)$ with respect to balanced colouring.
        For the Borsuk signed graph on $S^d$, we obtain balanced chromatic number $d+1$ for sufficiently small distance parameter and relate this equality to antipodal covers of the sphere.
        On $S^2$, spherical and projective covering bounds give three, four and six as the first successive values of the balanced chromatic number.

	\end{abstract}

	\section{Introduction}

	A \emph{signed graph} $(G,\sigma)$ is a graph $G=(V,E)$ endowed with a \emph{signature function} $\sigma: E(G)\rightarrow \{-1, +1\}$ which assigns to each edge $e$ a sign $\sigma(e)$. An edge $e$ is called a {\em positive edge} (or {\em negative edge}, respectively) if $\sigma(e) = +1$ (or $\sigma(e) = -1$, respectively). The graph $G$ is called the {\em underlying graph} of $(G,\sigma)$.

	\begin{definition}
		\label{def-switching}
		Assume $(G, \sigma)$ is a signed graph and $v$ is a vertex of $G$. The operation \emph{vertex switching} of $v\in V(G)$ results in a signature $\sigma'$
		defined as
		\[
		\sigma'(e) = \begin{cases}
			- \sigma(e), &\text{if $v$ is a vertex of $e$ and $e$ is not a loop}; \cr
			\sigma(e), &\text{ otherwise}.
		\end{cases}
		\]
		Two signatures $\sigma_1$ and $\sigma_2$ on the same underlying graph $G$ are said to be \emph{switching equivalent}, denoted by $\sigma_1\equiv\sigma_2$, if the set of edges on which they differ consists exactly of the edges between $X$ and $V(G)\setminus X$ for some $X\subseteq V(G)$.
	\end{definition}

Switching a set $X\subseteq V(G)$ reverses the signs of the edges between $X$ and $V(G)\setminus X$. For finite $X$, this amounts to switching its vertices in any order. Equivalently, if $s$ is $-1$ on $X$ and $+1$ elsewhere, the new sign of an edge $e=uv$ is $s(u)\sigma(e)s(v)$.

	Given a graph $G$, we denote by $(G,+)$ ($(G,-)$, respectively) the signed graph whose signature function is constantly positive (negative, respectively) on $G$.

	\begin{definition}
		\label{def-balanced}
		A signed graph $(G, \sigma)$ is {\em balanced} if its signature is switching equivalent to the all-positive signature. A subset $X$ of vertices of a signed graph $(G, \sigma)$ is called balanced if the signed subgraph induced by $X$ is balanced.
	\end{definition}

	Note that switching does not change the parity of the number of negative edges in a cycle, and a signed cycle $(C, \sigma)$ is balanced if it has an even number of negative edges, or equivalently, $\prod_{e \in E(C)}\sigma(e) =1$. If a signed graph $(G, \sigma)$ is balanced then every cycle must be balanced. Harary \cite{Harary1953} proved that this necessary condition is also sufficient.

	\begin{definition}
		\label{def-balancedcolouring}
		Assume $(G, \sigma)$ is a signed graph and $p$ is a nonnegative integer. A {\em balanced $p$-colouring} of $(G, \sigma)$ is a mapping $f: V(G) \to [p]$ such that for each colour $i$, the set $f^{-1}(i)$ is a balanced set of $V(G)$. The {\em balanced chromatic number} of $(G, \sigma)$ is defined as
		$$\chi_b(G, \sigma) = \min\{p:\text{ there is a balanced $p$-colouring of $(G, \sigma)$}\}.$$
	\end{definition}
We take $[0]=\varnothing$ and $\min\varnothing=\infty$. The empty graph has balanced chromatic number zero. A finite signed graph admits a balanced colouring if and only if it has no negative loop. Positive loops do not affect balance. We include them at every vertex of our Kneser, Schrijver and Borsuk signed graphs.

    Zaslavsky~\cite{Zaslavsky1987} called $\chi_b(G,\sigma)$ the \emph{balanced partition number}. It is closely related to the ``zero-free chromatic number'' or ``strict chromatic number'' defined in \cite{Zaslavsky1982-DM}. In particular, a \emph{zero-free $p$-colouring} of $(G,\sigma)$ is a mapping $c:V(G)\to\{\pm 1, \pm 2,\ldots,\pm p\}$ such that $c(x)\neq \sigma(e)c(y)$ for every edge $e=xy$ of $G$. The \emph{zero-free chromatic number}, $\chi^*(G,\sigma)$ is thus the minimum $p$ such that $(G,\sigma)$ admits a zero-free $p$-colouring.

    In a zero-free colouring $c$ of $(G,-\sigma)$, switching by the signs of $c(v)$ makes every edge of $(G,\sigma)$ within a class of constant $|c(v)|$ positive. Conversely, given a balanced colouring of $(G,\sigma)$, multiply each vertex colour by its switching sign within its colour class to obtain a zero-free colouring of $(G,-\sigma)$. Thus $\chi_b(G,\sigma)=\chi^*(G,-\sigma)$, as in Zaslavsky's characterization by antibalanced classes~\cite[Theorem~1(a)]{Zaslavsky1984}.

    A signed graph has a balanced $1$-colouring if and only if every cycle is positive. It has a zero-free $1$-colouring if and only if every even cycle is positive and every odd cycle is negative.

In the Borsuk signed graph a negative closed walk gives a loop in real projective space whose lift to the sphere joins antipodal points (Section~\ref{Borsuk signed}). This correspondence motivates our use of balanced colouring.

Contracting a positive non-loop edge of a signed cycle preserves the sign of the cycle. For the role of balanced colouring in signed analogues of Hadwiger's conjecture, see~\cite{JMNNQ24+}.

    We further remark on the following two connections between the balanced chromatic number of signed graphs and the classic chromatic number of graphs, which are easy consequences of the corresponding properties of zero-free colouring (see~\cite{Zaslavsky1982-DM}). Denote by $(G, \pm)$ the signed graph obtained from $G$ by replacing each edge $e=xy$ with a pair of parallel edges of opposite signs.
	\begin{proposition}\label{prop:Xb_NegativeGraph}
		For every finite loopless graph $G$, $\chi_b(G,-)=\lceil\chi(G)/2\rceil$ and $\chi_b(G,\pm)=\chi(G)$.
	\end{proposition}

    The second equality recovers ordinary graph colouring as a special case of balanced colouring.

    For example, as a generalization of the Four Colour Theorem to signed graphs, M\'{a}\v{c}ajov\'{a},  Raspaud and \v{S}koviera \cite{m2016chromatic} conjectured that every signature on a simple planar graph admits a zero-free colouring with the four colours $\{\pm1,\pm2\}$. That is equivalent to claiming that every signature on a simple planar graph admits a balanced 2-colouring. The conjecture was refuted by Kardo\v{s} and Narboni \cite{KN21}.

For a positive integer $n$, let $[n] = \{1,2,\ldots, n\}$. Denote by $\binom{[n]}{k}$ the set of all $k$-subsets of $[n]$. For $n\ge 2k$, the Kneser graph $K(n,k)$ has vertex set $\binom{[n]}{k}$, in which two vertices are adjacent if they are disjoint $k$-subsets of $[n]$.  It was conjectured by Kneser \cite{Kneser1955} and proved by Lov\'{a}sz \cite{MR0514625}  that the chromatic number of $K(n,k)$ is $n-2k+2$.
	Schrijver graph $S(n,k)$ is the subgraph of $K(n,k)$ induced by the set of stable  $k$-subsets, where a $k$-subset $A$ of $[n]$ is {\em stable} if $i \in A$ implies $i+1 \notin A$, where $i\in [n-1]$, and $n \in A$ implies that $1 \notin A$. It was proved by Schrijver \cite{Schrijver} that $S(n,k)$ is a vertex-critical subgraph of $K(n,k)$, i.e., $\chi(S(n,k)) = \chi(K(n,k))=n-2k+2$ and for any vertex $A$ of $S(n,k)$, $\chi(S(n,k)-A) =n-2k+1$.

    Lov\'{a}sz's proof of Kneser's conjecture initiated the application of topological methods in graph colouring. B\'ar\'any~\cite{Barany1978} gave a geometric proof, and Ziegler~\cite{Ziegler2002} gave combinatorial proofs of several Kneser colouring theorems. For a comparison of topological lower bounds, see Matou\v{s}ek and Ziegler~\cite{MatousekZiegler2004}.

Here we identify signed analogues of the Kneser and Schrijver graphs, determine their balanced chromatic numbers, and use the classical hemisphere method to prove vertex-criticality for the subgraph obtained by choosing one member of each antipodal pair in the alternating family.

	Throughout, $1\leq k\leq n$ are integers. Let $ \pm [n] =\{\pm 1, \pm 2, \ldots, \pm n\}$. A \emph{signed $k$-subset} of $[n]$ is a $k$-subset $A$ of $\pm [n]$ such that for any $i \in [n]$, $|A \cap \{i, -i\} | \le 1$. We denote by $\binom{[n]}{\pm k}$ the set of all signed $k$-subsets of $[n]$. For $A \in  \binom{[n]}{\pm k}$, let $-A = \{ -a: a \in A\}$. Thus a $k$-subset of $\pm [n]$ is a signed $k$-subset of $[n]$ if and only if $A\cap (-A)=\emptyset$. A signed $k$-subset of $[n]$ can naturally be represented by a $\{-1,0,1\}$-vector of length $n$ whose coordinates are labeled by $[n]$ and whose number of nonzero coordinates is $k$.

	\begin{definition}
		\label{def-signedKneser}
		The Kneser signed graph ${\KS}(n,k)$ has $\binom{[n]}{\pm k}$ as the vertex set where $A,B$ are joined by a positive edge if $A \cap (-B) = \emptyset$, and  $A, B$ are joined by a negative edge if $A \cap B = \emptyset$.
	\end{definition}

	Viewing vertices as vectors, vertices $A$ and $B$ are adjacent by a positive (respectively, negative) edge if the coordinatewise product is non-negative (respectively non-positive).

    For integers $p\ge q\ge1$, a balanced $(p,q)$-colouring assigns each vertex $q$ distinct colours from $[p]$, with each colour class balanced. The fractional balanced chromatic number is the infimum of $p/q$ over such colourings. Such a colouring exists precisely when the vertices can be mapped to $\KS(p,q)$ so that every edge maps to an edge of the same sign~\cite[Proposition~4.2]{KNWYZZ25}.

    In this paper, we study the balanced colouring of Kneser signed graphs and prove the following result.

	\begin{theorem}\label{thm-signedK}
		For any positive integers $n \ge k \ge 1$, $$\chi_b(\KS(n,k)) = n-k+1.$$
	\end{theorem}

	\begin{definition}
		\label{def-signedSchrijver}
		For $A=\{a_1<\cdots<a_k\}\in\binom{[n]}{k}$, put
		\[
		    T_A=\{(-1)^{j-1}a_j:1\le j\le k\}.
		\]
		The \emph{alternating} signed $k$-subsets with support $A$ are $T_A$ and $-T_A$. Denote by $\A(n,k)$ the family of alternating signed $k$-subsets of $[n]$. The {\em Schrijver signed graph} $\SSG(n,k)$ is the subgraph of $\KS(n,k)$ induced by the vertex set $\A(n,k)$.

	\end{definition}

	In terms of vectors, $\A(n,k)$ consists of those vertices of $\KS(n,k)$ whose nonzero entries are alternating.

	Let $\HKS(n,k)$ be the subgraph of $\KS(n,k)$ induced by the set of vertices whose first nonzero coordinate is positive. Likewise, $\HSS(n,k)$ is induced in $\SSG(n,k)$ by the alternating sets whose first nonzero coordinate is positive.

    Observe that replacing $A$ with $-A$ in $\HKS(n,k)$ is the same as switching the vertex $A$. A balanced colouring extends to the omitted representatives by giving $-A$ the colour of $A$ and the opposite switching sign within that colour class. Restricting a balanced colouring gives the converse. Thus $\KS(n,k)$ and $\HKS(n,k)$ have the same balanced chromatic number, as do $\SSG(n,k)$ and $\HSS(n,k)$. Next, we shall prove that $\HSS(n,k)$ is a vertex-critical subgraph of $\KS(n,k)$.

	\begin{theorem}
		\label{thm:signedS}
		For any positive integers $n \ge k \ge 1$, $$\chi_b(\HSS(n,k)) = n-k+1.$$ Moreover, for any vertex $A$ of $\HSS(n,k)$,
		$\HSS(n,k) - A$ admits a balanced $(n-k)$-colouring.
	\end{theorem}

There is also an ordinary graph formulation of \Cref{thm:signedS}. For each $A\in\binom{[n]}{k}$, choose $R_A\in\{T_A,-T_A\}$. Let $G_R$ be the graph on these supports in which $A$ and $B$ are adjacent when $R_A\cap R_B=\varnothing$. This is the negative subgraph of $\HSS(n,k)$ after switching the vertices for which $R_A=-T_A$. Each class of a proper colouring of $G_R$ has only positive edges in this switching, so \Cref{thm:signedS} gives $\chi(G_R)\ge n-k+1$ for every choice of representatives. For $R_A=T_A$, colouring $A$ by $\min A$ attains equality. After deleting a prescribed support $A_0$, choose for each remaining $A$ a colour $j\in A\setminus A_0$ and take $R_A$ to contain $+j$. This gives a choice of representatives for which $G_R-A_0$ needs at most $n-k$ colours.

\section{Balanced colouring of Kneser and Schrijver signed graphs}\label{sec:BalColKneser}

For $i\in[n]$, let $\B_i(n,k)=\{A\in\binom{[n]}{\pm k}:A\cap\{i,-i\}\ne\emptyset\}$.
This set is balanced in $\KS(n,k)$. Indeed, switching the vertices containing $-i$ makes every edge in the set positive. Any $n-k+1$ of these sets cover the graph, since only $k-1$ absolute values are left.
Hence $\chi_b(\HKS(n,k))\le n-k+1$.
We shall prove that $\chi_b(\HSS(n,k))\ge n-k+1$, which would imply that $\chi_b(\HSS(n,k))=\chi_b(\HKS(n,k))=n-k+1$.
For the full Kneser signed graph, there is also a direct proof of the lower bound from Schrijver's theorem.

\begin{proof}[Proof of the lower bound for Theorem~\ref{thm-signedK}]
We order the elements of $\pm[n]$ in cyclic order as $(1,-1,2,-2,\ldots,n,-n)$.
A stable $k$-set in this order contains at most one of $i,-i$ for every $i$. Two disjoint stable $k$-sets are joined by a negative edge in $\KS(n,k)$.
Thus $(S(2n,k),-)$ is a subgraph of $\KS(n,k)$. It follows from Proposition~\ref{prop:Xb_NegativeGraph} that $2n-2k+2=\chi(S(2n,k))\le 2\chi_b(\KS(n,k))$. Hence, $\chi_b(\KS(n,k))\ge n-k+1$.
\end{proof}

\subsection{Proof of Theorem \ref{thm:signedS}}

Fix $A\in V(\HSS(n,k))$. Observe that the only alternating $k$-sets contained in $A\cup -A$ are $A$ and $-A$ themselves. Therefore, the collection $\{\B_i(n,k):\{i,-i\}\cap A=\emptyset\}$ of $n-k$ balanced sets covers all vertices of $\HSS(n,k)$ except $A$. Hence
\[
  \chi_b(\HSS(n,k)-A)\le n-k.
\]
For the lower bound, we use open hemispheres to find a monochromatic negative closed walk in any putative balanced $(n-k)$-colouring.

Let $S^d:=\{x\in\R^{d+1}:\Vert x\Vert_2=1\}$ be the $d$-dimensional sphere. We say a subset $C\subseteq S^d$ is \emph{(antipodally) symmetric} if $-C=C$.
We use the following covering form of Ky Fan's theorem~\cite{Fan1952}, as stated in~\cite[Section~3.3]{ST06}. For its extension to oriented matroids, see~\cite{Zivaljevic2010}.

\begin{theorem}\label{thm:KyFan}
Let $(A_i)_{i\in I}$ be a linearly ordered family of open subsets of $S^d$, or a finite linearly ordered family of closed subsets. Suppose $A_i\cap(-A_i)=\varnothing$ for all $i$ and $\bigcup_{i\in I}(A_i\cup(-A_i))=S^d$. Then there are indices $i_1<\cdots<i_{d+1}$ and a point $x\in\bigcap_{j=1}^{d+1}(-1)^j A_{i_j}$.
\end{theorem}

All components below are connected components. Since $S^d$ is locally path-connected, the components of its open subsets are open and path-connected~\cite{Munkers2000}. The following is the component formulation of Bacon's open-cover theorem~\cite{Bacon66}.
\begin{theorem}\label{thm:signedBorsukUlam}
For every cover of $S^d$ by $d$ symmetric open sets, some member has a component containing a pair of antipodal points.
\end{theorem}
\begin{proof}
Let $C_1,\ldots,C_d$ be such a cover. Suppose no component of any $C_i$ contains an antipodal pair. For each $i$, the components of $C_i$ then occur in disjoint pairs $X,-X$. Let $U_i$ be the union of one component from each such pair. Then $U_i$ is open, $U_i\cap(-U_i)=\emptyset$, and $U_i\cup(-U_i)=C_i$. Applying \Cref{thm:KyFan} to $U_1,\ldots,U_d$ would require $d+1$ distinct indices.
\end{proof}

We use the alternately reoriented moment curve from Schrijver's proof \cite{Schrijver}, with the antipodal points adjoined. For $k\ge2$, the positive-index vectors below are, up to positive rescaling and a common sign, a Gale transform of the $n$ labelled points on the moment curve in $\R^{k-2}$ at parameters $1^2,\ldots,n^2$. The hemisphere property is the dual form of the alternating signs in a minimal affine dependence on that curve \cite{Gale1956,Breen1973,CordovilDuchet2000}. We include the direct verification.

\begin{theorem}\label{thm:signedGaleSchrijver}
There is a placement of $\pm[n]$ in the sphere $S^{n-k}$ such that the images of $i$ and $-i$ are antipodal for each $i\in[n]$, and any open hemisphere contains the images of all labels of an alternating $k$-set. The placement is injective when $n>k$.
\end{theorem}
\begin{proof}
Let $d=n-k$. On the odd moment curve take the vectors
\[
v_i=(-1)^i(i,i^3,\ldots,i^{2d+1})\in\R^{d+1}
\]
for each $i\in\pm[n]$. By definition $v_{-i}=-v_i$ for all $i$.
Place label $i$ at $w_i=v_i/\|v_i\|_2$. If $d\ge1$, the ratio of the second coordinate to the first determines $|i|$, so the placement is injective.

Put $t_i=i^2$ for $i\in[n]$. For an open hemisphere $H_a=\{w\in S^d:a\cdot w>0\}$, the sign of $a\cdot v_i$ at a positive index $i$ is the sign of $(-1)^i q(t_i)$, where $q$ is a nonzero polynomial of degree at most $d$. We may reduce to the case where $q$ vanishes at $d$ of the $t_i$. If $m<d$ samples vanish, the polynomials of degree at most $d$ vanishing there form a space of dimension $d+1-m\ge2$. Choose $r$ in this space independent of $q$, and among the ratios $-q(t_i)/r(t_i)$ with $r(t_i)\ne0$ take one of smallest absolute value, say $s$. Such a ratio exists since $n>d$, and it is nonzero. The polynomial $q+sr$ retains the old zeros, vanishes at a further sample and keeps the signs at all surviving samples. It is nonzero by independence. Repeat until $d$ samples vanish.

There are now $k$ surviving indices. For each such $i$, take label $i$ if $(-1)^iq(t_i)>0$ and $-i$ otherwise, using the final polynomial $q$. These labels lie in the original hemisphere. The $d$ sampled roots of the final polynomial $q$ exhaust its degree, so they are simple. For consecutive surviving indices $i<j$ there are exactly $j-i-1$ roots between $t_i$ and $t_j$. The sign of $q$ changes by the factor $(-1)^{j-i-1}$, while the prefactor changes by $(-1)^{j-i}$. The chosen labels therefore alternate.
\end{proof}

\begin{proof}[Proof of the lower bound for Theorem~\ref{thm:signedS}]
When $n=k$, $\HSS(n,k)$ has one vertex. Suppose $d=n-k>0$ and $\HSS(n,k)$ admits a balanced $d$-colouring. Extend it to a balanced colouring $f$ of $\SSG(n,k)$ by giving $X$ and $-X$ the same colour and opposite switching signs, as in the introduction.

Use the placement of \Cref{thm:signedGaleSchrijver} and write $w(X)=\{w_j:j\in X\}$. For each alternating $k$-set $X$ let $U_X=\{x\in S^d:w(X)\subseteq H_x\}$. This set is open, being defined by the strict inequalities $x\cdot w_j>0$ for $j\in X$. Put $A_i=\bigcup_{f(X)=i}U_X$. Since $U_{-X}=-U_X$ and $f(-X)=f(X)$, each $A_i$ is symmetric. The hemisphere property implies that these $d$ open sets cover $S^d$.

By \Cref{thm:signedBorsukUlam}, a component $C$ of some $A_i$ contains $x_0$ and $-x_0$. The nonempty sets $C\cap U_X$ with $f(X)=i$ form a relatively open cover of $C$. Their intersection graph is connected, since a partition into two nonempty families with no intersections between them would separate $C$. Thus there is a chain $U_{X_0},\ldots,U_{X_m}$ from a set containing $x_0$ to one containing $-x_0$, with consecutive sets meeting in $C$. Consecutive label sets lie in a common open hemisphere and hence satisfy $X_j\cap(-X_{j+1})=\varnothing$, giving a positive edge. The first and last lie in opposite open hemispheres, so $X_m\cap X_0=\varnothing$, giving a negative edge back to $X_0$. This closed walk in colour $i$ has negative sign. Switching that colour class to all positive preserves the sign of each closed walk, a contradiction.
\end{proof}

\subsection{Negative subgraphs of Kneser and Schrijver signed graphs}

For a signed graph $(G,\sigma)$, let $(G,\sigma)^-$ denote the spanning subgraph consisting of its negative edges. The following proposition is Zaslavsky's switching characterization~\cite[Theorem~1(b)]{Zaslavsky1984} applied to $(G,-\sigma)$.

\begin{proposition}\label{prop:Xb_NegativeSubgraph}
    For every finite signed graph $(G,\sigma)$ without negative loops,
    $$\chi_b(G,\sigma)=\min_{\sigma'\equiv\sigma}\chi((G,\sigma')^-).$$
\end{proposition}

The signatures inherited by $\HKS(n,k)$ and $\HSS(n,k)$ attain this minimum.

\begin{theorem}\label{thm:negsubgraphHKSHSS}
    For all $n\geq k$, $$\chi(\HKS(n,k)^-)=\chi(\HSS(n,k)^-)=n-k+1.$$
\end{theorem}
\begin{proof}
    The lower bound follows from \Cref{prop:Xb_NegativeSubgraph}, \Cref{thm-signedK}, and \Cref{thm:signedS}. Colour each vertex of $\HKS(n,k)^-$ by the index of its first nonzero coordinate. This coordinate is positive and its index is at most $n-k+1$. Two vertices of the same colour thus share a signed element.
\end{proof}

We next consider the negative subgraphs on the full vertex sets.

\begin{theorem}\label{thm:negsubgraphKSSS}
    For all $n\geq k$, $$\chi(\KS(n,k)^-)=2n-2k+2.$$
    For the Schrijver signed graph,
        \[
        \chi(\SSG(n,k)^-)=
        \begin{cases}
            2n, & k=1,\\
            n-k+2, & 2\le k\le n.
        \end{cases}
    \]
\end{theorem}
\begin{proof}
    The two antipodal copies of $\HKS(n,k)^-$ in $\KS(n,k)^-$ can be coloured with disjoint palettes of $n-k+1$ colours by \Cref{thm:negsubgraphHKSHSS}. The matching lower bound follows from the copy of $S(2n,k)$ in the proof of \Cref{thm-signedK}.

    For $k=1$, the $2n$ signed singletons are pairwise disjoint. For $k\ge2$, one of the first two nonzero entries of each alternating $k$-set is positive, and its index is at most $n-k+2$. Colouring by this index gives the upper bound.

    For the lower bound we use B\'ar\'any's hemisphere argument~\cite{Barany1978} with the configuration of \Cref{thm:signedGaleSchrijver}. Write $d=n-k$. Given a proper $(d+1)$-colouring $c$ of $\SSG(n,k)^-$, let $U_i$ consist of those $x\in S^d$ for which $w(X)\subseteq H_x$ for some alternating $k$-set $X$ of colour $i$. These sets are open, and they cover $S^d$ by \Cref{thm:signedGaleSchrijver}. By \Cref{thm:Borsuk-Ulam}, some $U_i$ contains $x$ and $-x$. The corresponding sets of colour $i$ lie in disjoint hemispheres, so they are joined by a negative edge, a contradiction.
\end{proof}

For $k=2$, R\'edei's theorem on tournaments gives a negative clique in every switching~\cite{Redei1934}.

\begin{theorem}\label{thm:signedSchrijverCase2}
    For $n\ge2$, the negative subgraph of every switching of $\HSS(n,2)$ contains $K_{n-1}$.
\end{theorem}
\begin{proof}
    Orient each edge $ij$ of $K_n$ from $i$ to $j$ when the chosen representative is $\{i,-j\}$. By R\'edei's theorem this tournament has a directed Hamiltonian path $v_1\to\cdots\to v_n$. The $n-1$ representatives $\{v_i,-v_{i+1}\}$ along the path are pairwise disjoint and hence form a negative clique.
\end{proof}

\section{Borsuk signed graphs}\label{Borsuk signed}

We recall the open-cover form of the Borsuk--Ulam theorem.
\begin{theorem}\label{thm:Borsuk-Ulam}
For any open cover $A_1, A_2, \ldots, A_{d+1}$ of $S^{d}$, one of the $A_i$'s contains a pair of antipodal points.
\end{theorem}

For $d\geq 1$ and $0<\eps<2$, the \emph{Borsuk graph} $B(d,\eps)$ has vertex set $S^d$ and edges $xy$ whenever $\|x+y\|\leq\eps$. We use Euclidean distance $\dist$ throughout. The chromatic number of $B(d,\eps)$ gives an equivalent form of the theorem~\cite[Section~4.4]{ST06}.
\begin{theorem}[Reformulation of Borsuk-Ulam]\label{thm:BorsukColoring}
For each $d\geq 1$ there is an $\eps_d>0$ such that $\chi(B(d,\eps))=d+2$ for $0<\eps\leq\eps_d$.
\end{theorem}

For $n>2k$, $d=n-2k$ and sufficiently small $\eps$, Simonyi and Tardos~\cite[Proposition~8 and Lemma~4.4]{ST06} give a homomorphism $B(d,\eps)\to S(n,k)$. We consider the signed analogue.
\begin{definition}
For $d\ge0$ and $0<\eps<2$, the \emph{Borsuk signed graph} $\BS(d,\eps)$ is the signed graph on the vertex set $S^d$, where for any $x,y\in S^d$, there is a positive edge joining $x$ and $y$ if $\dist(x,y)\leq \eps$, and a negative edge joining them if $\dist(x,-y)\leq \eps$.
\end{definition}

For every $0<\eps<2$ this signed graph has finite balanced chromatic number. The caps $U_z=\{x\in S^d:z\cdot x>\eps/2\}$, $z\in S^d$, cover the sphere. Each cap is balanced because $x,y\in U_z$ implies $\|x+y\|\ge z\cdot(x+y)>\eps$. A finite subcover gives a balanced colouring by assigning each point to the first cap containing it.

Bacon's antipodal covering theorem is formulated for symmetric sets that split into two disjoint relatively open antipodal halves~\cite[Section~2]{Bacon66}. We give the component formulations and their relation to balanced colouring.
\begin{theorem}\label{thm-signedBorsuk}
For every $d\geq1$ the following statements are equivalent.
\begin{itemize}
\item[(a)] There exists $\eps_d>0$ such that $\chi_b(\BS(d,\eps))=d+1$ whenever $0<\eps\leq\eps_d$.
\item[(b)] In every cover of $S^d$ by $d$ symmetric open sets, some member has a component containing a pair of antipodal points \textup{(\Cref{thm:signedBorsukUlam})}.
\item[(c1)] In every cover of $S^d$ by $d$ symmetric closed sets, some member $F$ has the following property. Every open neighbourhood of $F$ has a component containing a pair of antipodal points.
\item[(c2)] In every cover of $S^d$ by $d$ symmetric closed sets, some member has a component containing a pair of antipodal points.
\end{itemize}
\end{theorem}

\begin{proof}
A set $F$ is balanced in $\BS(d,\eps)$ precisely when it admits switching signs $s(x)\in\{\pm1\}$ such that $T=\{s(x)x:x\in F\}$ satisfies $\|u+v\|>\eps$ for all $u,v\in T$. Indeed, a violation corresponds to a negative edge between vertices with equal signs or a positive edge between vertices with opposite signs.

For the upper bound, colour $x$ by the least index $i$ maximizing $|x_i|$, and take $s(x)$ to be the sign of $x_i$. Each switched point of colour $i$ has $i$th coordinate at least $1/\sqrt{d+1}$, so this colouring is balanced for $0<\eps<2/\sqrt{d+1}$. For the lower bound, assume (b). A balanced $d$-colouring would give switched sets $T_1,\ldots,T_d$ as above. The open $\eps/3$-neighbourhood $V_i$ of $T_i$ is disjoint from $-V_i$ by the triangle inequality. The symmetric open sets $V_i\cup(-V_i)$ cover $S^d$, and every connected subset of one of them lies in a single half. This contradicts (b), proving the lower bound for every $0<\eps<2$ and hence (a).

For (a)$\Rightarrow$(c1), suppose each member $F_i$ of a symmetric closed cover has an open neighbourhood $U_i$ in which no component contains antipodes. Replace $U_i$ by $U_i\cap(-U_i)$. Its components are open and occur in antipodal pairs. The union $V_i$ of one component from each pair and its antipode $-V_i$ partition $U_i$ into open halves. Thus $K_i=F_i\cap V_i$ and $-K_i$ are relatively closed in $F_i$, hence compact. For each nonempty $F_i$ their distance is positive. Choose a common $0<\eps\le\eps_d$ smaller than all these distances. Switching points of $-K_i$ into $K_i$ makes $F_i$ balanced by the criterion above. Giving each point the first colour $i$ for which it belongs to $F_i$ now contradicts (a).

A finite open cover of the sphere has a closed shrinking~\cite{AH45}. For a symmetric cover, adjoining the antipodes to each closed set gives a symmetric closed shrinking. Applying (c1) to this shrinking gives (b), since each original open set is a neighbourhood of its closed member. Also, (c2) implies (c1), since a connected subset stays connected in every neighbourhood.

For (c1)$\Rightarrow$(c2), take $F$ as in (c1). For each $j\ge1$, some component $C_j$ of the open $1/j$-neighbourhood of $F$ contains points $x_j$ and $-x_j$. Pass to a convergent subsequence $x_j\to x$. Since $F$ is closed and $\dist(x_j,F)<1/j$, we have $x\in F$. For $m\ge1$ let $D_m$ be the component of $x$ in $E_m=\{y\in S^d:\dist(y,F)\le 1/m\}$. The spherical balls of Euclidean radius $1/m$ centred at $x$ and $-x$ are connected and lie in $E_m$. For sufficiently large $j$, $C_j\subseteq E_m$ meets both balls at $x_j$ and $-x_j$. Their union with $C_j$ is connected, so $-x\in D_m$. The sets $D_m$ are nested, compact and connected. Hence $\bigcap_m D_m$ is a connected subset of $F$ containing $x$ and $-x$, as required.
\end{proof}

On $S^2$, the switching criterion relates balanced colouring to covers of the sphere and of the real projective plane.

\begin{corollary}[The two-sphere]\label{thm:borsuk-two-sphere}
For $0<\eps\le\sqrt{8/3}$, the balanced chromatic number of the Borsuk signed graph on $S^2$ is
\[
\chi_b(\BS(2,\eps))=
\begin{cases}
3 & \text{if } 0<\eps\le 2/\sqrt{3},\\
4 & \text{if } 2/\sqrt{3}<\eps\le\sqrt{2},\\
6 & \text{if } \sqrt{2}<\eps\le\sqrt{8/3}.
\end{cases}
\]
For every $0<\eps<2$, $\chi_b(\BS(2,\eps))\ne 5$.
\end{corollary}
\begin{proof}
By \Cref{thm-signedBorsuk}(a) and monotonicity in $\eps$, at least three colours are needed. Put $c=\eps^2/2-1$. For the switched set $T_i$ of colour $i$, the switching criterion gives $u\cdot v>c$ for $u,v\in T_i$. The sets $T_i,-T_i$ cover $S^2$ and each has chordal diameter at most $\sqrt{4-\eps^2}$. In a six-set cover of $S^2$ some set has chordal diameter at least $\sqrt{8/3}$~\cite{Heppes1998}, so three colours force $\eps\le2/\sqrt3$. If $\eps>\sqrt2$, then $c>0$. The images of the $T_i$ under $u\mapsto uu^{\mathsf T}$ cover the rank-one orthogonal projectors on $\R^3$. Since $\|uu^{\mathsf T}-vv^{\mathsf T}\|_F^2=2-2(u\cdot v)^2$, each image has Frobenius diameter at most $\sqrt{2-2c^2}<\sqrt2$. Such a cover needs at least six sets~\cite[Lemma~2.4 and Proposition~4.1]{Projector2026}.

For the upper bounds, pair opposite radial faces of the cube, octahedron and dodecahedron, giving three, four and six colours, respectively. In an open cube face one signed coordinate exceeds $1/\sqrt3$, and in an open octant any two points have positive inner product~\cite{GodsilZaks2012}. For the dodecahedron, Arman, Bondarenko, Prymak and Radchenko~\cite[proof of Proposition~7]{ABPR2026} show that $u\cdot v\ge1/3$ on each closed radial face. The inequality is strict in the interior since $u\mapsto u\cdot v$ has its only local minimum on $S^2$ at $-v$. Thus, after switching into one face of each pair, $\|u+v\|$ strictly exceeds the respective endpoint for points in face interiors.

For each polyhedron a rotation takes any given finite vertex set off all face boundaries. The face colours therefore work on every finite set at the endpoint. Each edge imposes a closed constraint on its two labels (colour and switching sign), so compactness of the product of the finite label spaces gives a balanced colouring of the whole sphere.
\end{proof}

For $d=n-k\ge1$, place the labelled points $w_j$ on $S^d$ as in \Cref{thm:signedGaleSchrijver}. Let $m(x)=\max_A\min_{j\in A}x\cdot w_j$, where $A$ ranges over the alternating $k$-sets, and choose a maximizing set $A_x$. The function $m$ is continuous and positive on the sphere, so $m(x)\geq\delta>0$ uniformly. If $j\in A_x\cap(-A_y)$, then $(x-y)\cdot w_j\geq2\delta$. If $j\in A_x\cap A_y$, then $(x+y)\cdot w_j\geq2\delta$. Thus, for $0<\eps<2\delta$, the map $x\mapsto A_x$ sends each edge of $\BS(d,\eps)$ to an edge of the same sign in $\SSG(n,k)$. Choose $0<\eps<\min\{2\delta,\eps_d\}$ with $\eps_d$ as in \Cref{thm-signedBorsuk}(a). Pulling back a balanced colouring along this map gives $\chi_b(\SSG(n,k))\ge\chi_b(\BS(d,\eps))=n-k+1$.

The \emph{Lusternik--Schnirelmann category} of $X$ is the least $k$ for which $X$ has a cover by $k+1$ open sets contractible in $X$, meaning that each inclusion is null-homotopic~\cite{CTOT03}. The projective-space formulation is classical~\cite{CTOT03,Steinlein93}.
\begin{theorem}\label{thm:Lusternik-Schnirelmann}
The Lusternik-Schnirelmann category of the real projective space $\R\mathrm{P}^d$ is $d$.
\end{theorem}

For the double cover $q:S^d\to\RP^d$, an open set $U$ is contractible in $\RP^d$ if and only if $q$ admits a section over $U$. Lifting a homotopy from a constant map to the inclusion gives a section~\cite{Hatcher2001}. Conversely, if $s$ is a section and $u\in U$, its image misses $-s(u)$. Contracting the punctured sphere and composing with $q$ gives a null-homotopy of the inclusion. A section splits $q^{-1}(U)$ into the two open sets $s(U)$ and $-s(U)$. Conversely, if the components of $q^{-1}(U)$ occur in antipodal pairs, the union of one component from each pair is open and maps homeomorphically onto $U$. Thus (b) gives the category lower bound, and the $d+1$ standard affine charts give the upper bound.

This covering also gives a geometric interpretation of the sign of a closed walk in $\BS(d,\eps)$. For an edge $e=xy$, project the short spherical arc from $x$ to $\sigma(e)y$ to $\RP^d$. The arc is unique since $\|x-\sigma(e)y\|\le\eps<2$. Along a closed walk, lift successive arcs from the endpoint of the preceding lift. Each step multiplies the current sign by its edge sign, so the lift ends at the antipode of its starting point exactly when the walk is negative.

\noindent\textbf{Concluding remarks.} By \Cref{prop:Xb_NegativeSubgraph,thm:signedS}, the negative subgraph of every switching of $\HSS(n,k)$ needs at least $n-k+1$ colours. In an earlier version we conjectured that for $n>k$ every such negative subgraph contains $S(n-1,k/2)$ for even $k$ and $S(n,(k+1)/2)$ for odd $k$. An anonymous referee gave a counterexample in $\HSS(6,4)$ by switching the vertices with supports $2345$ and $2456$. The resulting negative subgraph has odd girth seven and hence cannot contain the predicted $S(5,2)\cong C_5$. Here $abcd$ denotes $\{a,b,c,d\}$. Deleting the edge between $1345$ and $3456$ leaves a bipartite graph in which its ends are at distance six.

The negative subgraph can even lack every Schrijver graph of the same chromatic number. In $\HSS(7,4)$ switch the vertices with supports $1246,1256,2346,3456$. Each resulting representative contains a positive element of $[4]$, so colouring by the least such element gives $\chi(G_R)=4$ by \Cref{thm:signedS,prop:Xb_NegativeSubgraph}.

The graph $G_R$ has five triangles. Two meet at $2356$, two others meet at $3456$, and there are no other intersections between distinct triangles. To check for a $K_{4,4}$, put $N(A)=\{B:R_A\cap R_B=\varnothing\}$. For a four-vertex set $Q$, count its common neighbours by intersecting the sets $N(A)$ for $A\in Q$. The counts $0$, $1$, $2$ and $3$ occur $49\,416$, $2\,513$, $338$ and $93$ times, respectively. This accounts for all $\binom{35}{4}$ choices of $Q$, so there is no $K_{4,4}$. The supplementary script \texttt{verify-hss74.py} lists the triangles and computes these counts from the representatives.

This excludes every $4$-chromatic Schrijver graph $S(2r+2,r)$. For $r=1,2$ such a graph has a triangle meeting three others (use the diameters $14,25,36$ for $r=2$). For $r\ge3$ the $r$-subsets of the odd and even positions form the two sides of a $K_{r+1,r+1}$.

On the two-sphere, \Cref{thm:borsuk-two-sphere} determines the balanced chromatic number up to $\eps=\sqrt{8/3}$. Which integers occur as $\chi_b(\BS(2,\eps))$ for $\sqrt{8/3}<\eps<2$, and for which values of $\eps$ is each attained?

\section*{Acknowledgment}

We thank an anonymous referee for the counterexample on $\HSS(6,4)$, for pointing out the application of Rédei's theorem, and for suggestions that improved the proofs and exposition.

    The second author has received support under the program ``Investissement d'Avenir'' launched by the French Government and implemented by ANR, with the reference ``ANR‐18‐IdEx‐0001'' as part of its program ``Emergence''. The third and sixth authors are supported by grant NSFC   12371359. The sixth author is also supported by grant U20A2068.


\end{document}